\documentclass[10pt,a4paper,reqno,pdftex]{amsart}

\usepackage{amsmath,amssymb,amscd} \usepackage{a4wide}  
\usepackage[utf8]{inputenc} \usepackage[english]{babel}
\usepackage{tikz}
\usetikzlibrary{decorations}
\usepackage{graphicx,subfigure}
\usepackage[font=footnotesize]{caption}
\usepackage{mwe}
\usepackage{xcolor}
\usepackage[basic]{mnotes}
\usepackage{cite}

\newtheorem{theorem}{Theorem}[section]
\newtheorem{lem}[theorem]{Lemma}
\newtheorem{cor}[theorem]{Corollary}
\newtheorem{prop}[theorem]{Proposition}

\usepackage{xfrac}
\usepackage{nicefrac} 
\usepackage{empheq}

\theoremstyle{definition}

\newtheorem{remark}[theorem]{Remark}

\newcommand{\T}{\mathbb{T}}
\newcommand{\R}{\mathbb{R}}
\newcommand{\Z}{\mathbb{Z}}
\newcommand{\N}{\mathbb{N}}
\renewcommand{\:}{\colon}
\renewcommand{\varphi}{\phi}
\newcommand{\ssq}{\subseteq}
\newcommand{\eps}{\varepsilon}
\newcommand{\mc}{\mathcal}

\makeatletter
\renewenvironment{proof}[1][\proofname]{%
   \par\pushQED{\qed}\normalfont%
   \topsep6\p@\@plus6\p@\relax
   \trivlist\item[\hskip\labelsep\bfseries#1\@addpunct{.}]%
   \ignorespaces
}{%
   \popQED\endtrivlist\@endpefalse
}
\makeatother

\begin{document}

\title[Pinched systems]{On the probability of positive finite-time Lyapunov exponents on strange nonchaotic attractors}

\author{F.~Remo}
\address{Institute of Mathematics, Friedrich Schiller University Jena, Germany}
\email{flavia.remo@uni-jena.de}

\author{G.~Fuhrmann}
\address{Department of Mathematical Sciences, Durham University, United Kingdom}
\email{gabriel.fuhrmann@durham.ac.uk}

\author{T.~J\"ager} \address{Institute of Mathematics, Friedrich Schiller
  University Jena, Germany} \email{tobias.jaeger@uni-jena.de}

\subjclass[2010]{}

\begin{abstract} 
We study strange non-chaotic attractors in a class of quasiperiodically forced
monotone interval maps known as pinched skew products. We prove that the
probability of positive time-$N$ Lyapunov exponents---with respect to the unique
physical measure on the attractor---decays exponentially as $N\to \infty$. The
motivation for this work comes from the study of finite-time Lyapunov exponents
as possible early-warning signals of critical transitions in the context of
forced dynamics.

\end{abstract}

\maketitle

\section{Introduction}\label{Intro}
In this article, we study quasiperiodically forced interval maps of the form
\begin{align}\label{eq: F intro}
F_{\kappa}:\mathbb{T}^{D}\times [0,1]\rightarrow \mathbb{T}^{D}\times [0,1],
\quad F_{\kappa}(\theta, x) = \left(\theta+v, \tanh(\kappa x)\cdot g(\theta)
\right),
\end{align}
where $\kappa>0$ is a real parameter, $v\in \T^D$ is a totally irrational
rotation vector\footnote{We say $v\in\T^D$ is \emph{totally irrational} if there is no  non-zero $n\in\Z^d$ with $\langle v,n\rangle\in \Z$.} and the
multiplicative forcing term $g:\T^D\to[0,1]$ is given by
\begin{align} \label{eq: forcing intro}
  g(\theta) = \frac{1}{D} \cdot\sum_{i=1}^{D}\sin(\pi\theta_{i}) \ . 
\end{align}
Systems of this kind are often called {\em pinched skew products}, where
{\em pinched} refers to the fact that the forcing term $g$ vanishes for some
$\theta\in\T^D$ (here, at $\theta=0$). 
Pinched skew-products received considerable attention due
the occurrence of so-called strange non-chaotic attractors (SNAs)
\cite{grebogi/ott/pelikan/yorke:1984,pikovski/feudel:1995,keller:1996,glendinning:2002,stark:2003,jaeger:2006,GroegerJaeger2013SNADimensions}. 
Due
to their specific properties---in particular, the pinching in combination with the
invariance of the zero line $\T^D\times\{0\}$---they are technically
more accessible than other forced systems that exhibit SNAs so that they have
been used on various occasions for case studies concerning the structural
properties of such attractors. This led, for instance, to first results on the
topological structure \cite{jaeger:2006} and the dimensions
\cite{GroegerJaeger2013SNADimensions} of SNAs, which have later been extended to
the more difficult situation of additive quasiperiodic forcing
\cite{bjerkloev:2005,bjerkloev:2007,FuhrmannGrogerJager2018}.

In a similar spirit, the aim of this note is to establish a quantitative result
on the distribution of positive finite-time Lyapunov exponents on the SNA
appearing in the system given by (\ref{eq: F intro}) and (\ref{eq: forcing
  intro}). Given $(\theta,x)\in\T^D\times[0,1]$ and $N\in\N$, we define the
time-$N$-Lyapunov exponent as
\begin{align*}
 \lambda_N(\theta,x)=\log \left(\partial_x F^{N}_\kappa(\theta,x)\right)/N \ .
\end{align*}
The (asymptotic) Lyapunov exponents are then given by
\begin{align*}
 \lambda(\theta,x)=\lim_{N\to\infty} \lambda_N(\theta,x) \ . 
\end{align*}
As established in \cite{keller:1996}, for any $\kappa>\kappa_0:=e^{-\int_{\T^D}
  \log g(\theta)d\theta}$, there exists a unique physical measure $\mathbb
P_\kappa$ of the system (\ref{eq: F intro}) that is ergodic and has a negative
Lyapunov exponent. As a consequence, asymptotic Lyapunov exponents are $\mathbb
P_\kappa$-almost surely negative. 
However, on the invariant zero line
$\T^D\times\{0\}$, the pointwise Lyapunov exponents almost surely equal
$\log\kappa-\log\kappa_0$ (see Remark~\ref{r.zero_line_exponent} below). Hence,
for $\kappa>\kappa_0$, positive asymptotic Lyapunov
exponents are still present in the system and lead to a positive probability of
positive finite-time exponents for all times $N\in\N$. Our main result provides
information on the scaling behaviour of these probabilities.

\begin{theorem}\label{thm: main}
 Denote by $\mathbb P_\kappa$ the unique physical measure of (\ref{eq: F intro})
 with forcing function~(\ref{eq: forcing intro}). Let $p_{\kappa,N}=\mathbb
 P_\kappa(\{(\theta,x)\mid \lambda_N(\theta,x)>0\})$. Then there exists
 $\kappa_1>\kappa_0$ such that for all $\kappa\geq \kappa_1$, there are constants
 $\gamma_+\geq \gamma_->0$ (depending on $\kappa$) such that
 \[
  \exp(-\gamma_+ N)\leq p_{\kappa,N}\leq \exp(-\gamma_- N)
  \]
  holds for all $N\in\N$.
\end{theorem}

\begin{figure}[b]
\centering \includegraphics[scale=0.35]{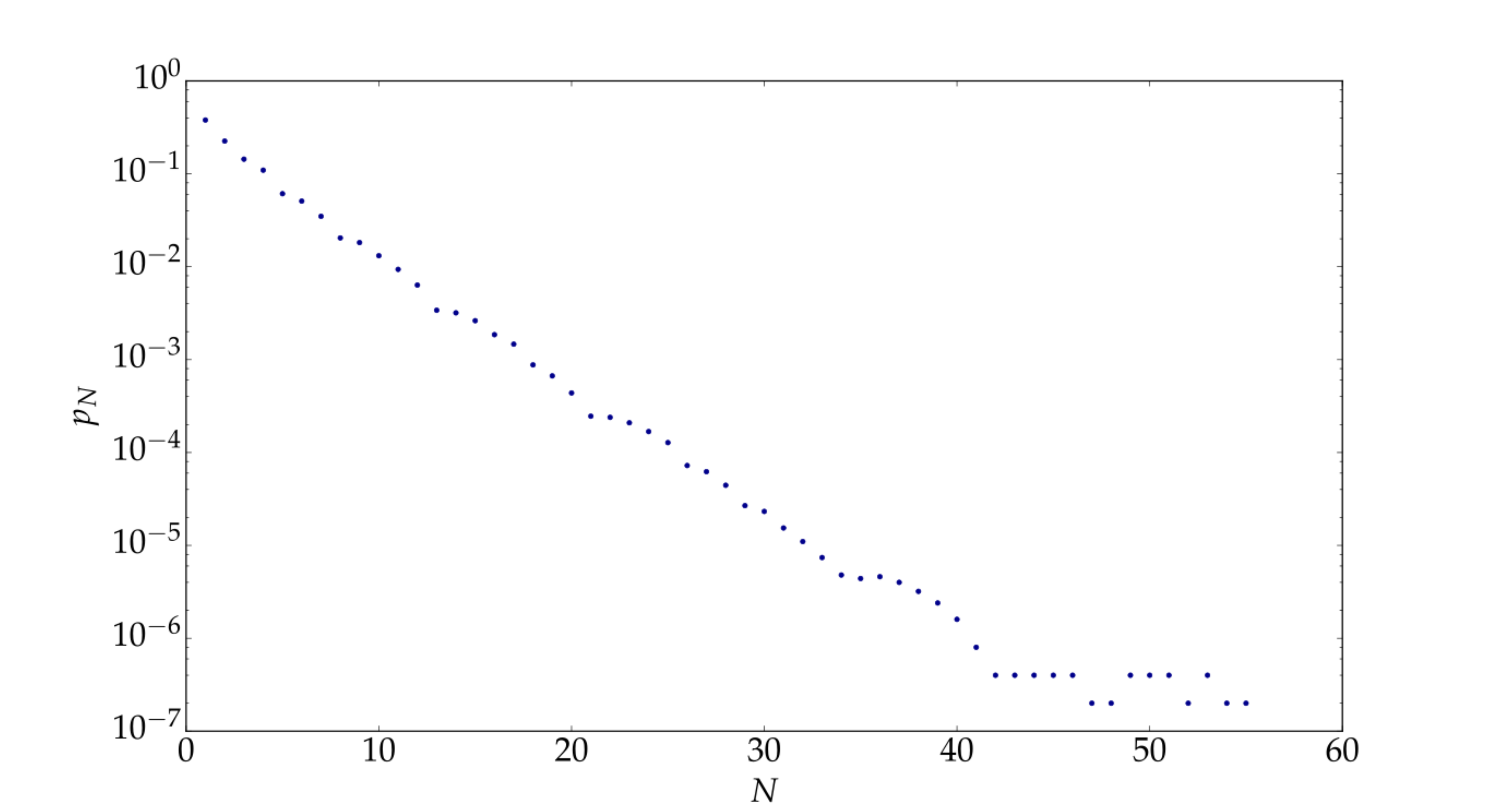}
 \caption[Scaling of probability of non-negative finite Lyapunov exponent]{A
   logarithmic plot of the numerically obtained probability $p_N=p_{\kappa,N}$ over
   $N$ for the system \eqref{eq: F intro} with $D=1$, $\kappa=3$ and $v$ the
   golden mean.  The graph shows the relative frequency of non-negative
   finite-time Lyapunov exponents among a grid of $5\cdot 10^6$ initial
   conditions on the SNA (see also Figure~\ref{fig: sna pinched}). Consistent
   with the statement of Theorem~\ref{thm: main}, the plot indicates an
   exponential decay.} \label{fig: FTLE}
\end{figure}

Apart from its intrinsic interest, motivation for this result stems from the
study of critical transitions. One major problem in this field is the
identification of suitable (that is, observable and reliable) early warning
signals
\cite{van2007slow,Schefferetal2009EWSforCT,scheffer2009critical,kuehn2011mathematical,scheffer2012anticipating,veraart2012recovery}
for such transitions.  A commonly proposed and utilized early warning signal for
fold bifurcations---which are often cited as a paradigmatic example of critical
transitions---are \emph{slow recovery rates} (also referred to as a
\emph{critical slowing down})
\cite{van2007slow,Schefferetal2009EWSforCT,scheffer2009critical,scheffer2012anticipating}. Since
this notion has been coined in an interdisciplinary context and is used in a
wide variety of situations, there is no comprehensive and rigorous mathematical
definition of this term and we refrain from attempting to give one
here. However, in the classical case of an autonomous fold bifurcation, recovery
rates can be identified with the Lyapunov exponents of the stable
equilibria. Thus, in this situation, critical slowing down simply refers to the
fact that when the stable and unstable equilibria involved in the bifurcation
approach each other and eventually merge at the critical parameter, the
resulting single fixed point is neutral, that is, it has exponent zero.

This picture changes significantly when a fold bifurcation takes place under the
influence of external quasiperiodic forcing. First of all, the resulting
\emph{non-autonomous} systems generally do not allow for fixed points.
Therefore, when carrying over ideas from an autonomous to a non-autonomous
setting, one needs an appropriate replacement.  In the present context, this
part is played by so-called \emph{invariant graphs} (see
Section~\ref{preliminaries}) and accordingly, non-autonomous fold bifurcations
occur as invariant graphs approach each other upon a change of system
parameters.  In stark contrast to autonomous fold bifurcations, this does not
necessarily yield neutral invariant graphs but may instead lead to a strange
non-chaotic attractor-repeller-pair
\cite{nunez/obaya:2007,AnagnostopoulouJaeger2012SaddleNodes} created at the
bifucation point. This alternative pattern is referred to as a {\em non-smooth
  saddle-node bifurcation}. Moreover, just as for pinched systems, under
suitable conditions, there exists a unique physical measure $\mathbb P$ which is
supported on the attractor and has a negative Lyapunov exponent (see
\cite{keller:1996,fuhrmann2013NonsmoothSaddleNodesI}).  However, this means that
Lyapunov exponents remain $\mathbb P$-almost surely negative and bounded away
from zero during a non-smooth saddle-node bifurcation (see
Section~\ref{preliminaries} for more details).

While this seems to rule out the viability of slow-recovery rates as early
warning signals for non-smooth fold bifurcations, one should bear in mind that
experiments never measure the actual Lyapunov exponent but rather approximations
of it. In other words, one rather measures finite-time Lyapunov exponents
instead of asymptotic ones. Since the presence of an SNA implies that positive
finite-time exponents occur with positive probability for any time $N\in \N$
\cite{pikovski/feudel:1995,RemoFuhrmannJager2019}, one may hence wonder whether
the observation of non-negative finite-time Lyapunov exponents can help to
detect an SNA in practice. However, if $N$ is chosen too small, then positive
time-$N$-exponents can be observed already far from a
bifurcation. Conversely, for large $N$, the probability of observing positive
exponents on this time-scale converges to zero since the unique physical
measure has a negative exponent. It is in this context that the scaling behaviour
of the probabilities of time-$N$-exponents with $N\to\infty$ becomes
important. Numerical studies for the quasiperiodically forced Allee model
performed in \cite{RemoFuhrmannJager2019} remained somewhat inconclusive, which
is partly explained by the fact that the simulation of continuous-time systems
is considerably more time-consuming than that of discrete-time systems.
The
exponential decay obtained in Theorem~\ref{thm: main} is an indication that very
large data sets may be required to detect this kind of early-warning signals in
practice. As mentioned before, this interpretation relies on the hypothesis that
quasiperiodically forced systems undergoing a saddle-node bifurcation---as
studied in \cite{RemoFuhrmannJager2019}---show a behaviour comparable to that of
pinched systems treated here. 
We expect that using techniques from
\cite{bjerkloev:2005,fuhrmann2013NonsmoothSaddleNodesI,FuhrmannGrogerJager2018},
similar statements can be established for non-pinched systems but this
would require a considerably more involved analysis due to the inherent technical
difficulties.

This article is organised as follows.  In the next section, we introduce some
technical background on forced monotone interval maps and their invariant
graphs.  There, we also describe the physical measure $\mathbb P$ from above in
more detail.  In Section~\ref{sec: pinched systems}, we specify the class of
pinched skew-products for which we prove (a more general version of) the above
theorem.  This proof and the full statement---Theorem~\ref{thm: lower bound on
  probabilities} and Theorem~\ref{thm: exponential decay wrt graph measure}
(which gives the upper bound and is the harder part)---are given in the final
section, Section~\ref{sec: proof of main result}.

\medskip

{\bf Acknowledgments.}\quad This project has received funding from the European
Union's Horizon 2020 research and innovation program under the Marie Sk\l
odowska-Curie grant agreements No 643073 and No 750865.  TJ acknowledges support
by a Heisenberg grant of the German Research Council (DFG grant OE 538/6-1).

\section{Forced monotone interval maps and invariant graphs}\label{preliminaries}
Throughout this note, we deal with {\em quasiperiodically forced (qpf) monotone
  interval maps}, that is, skew products of the form
\begin{equation}
    F:\mathbb{T}^{D}\times [0,1]\rightarrow \mathbb{T}^{D}\times [0,1], \quad 
    (\theta, x) \mapsto (\rho(\theta),~ F_{\theta}(x)),
    \label{skew_product}
\end{equation}
where $\T^D=\R^D/\Z^D$ is the $D$-dimensional torus (for some $D\geq 1$),
\[
\rho\:\T^D\to \T^D,\qquad \theta\mapsto \theta+v
\]
is a minimal rotation with a \emph{rotation vector} $v$ and for each $\theta\in
\T^D$, $F_\theta$ is a continuously differentiable non-decreasing map on $[0,1]$
such that $(\theta,x)\mapsto F_{\theta}'(x)$ is continuous.  It is customary to
refer to $(\T^D,\rho)$ as the \emph{forcing system} (defined on the \emph{base}
$\T^D$); the maps $F_\theta$ $(\theta\in \T^D)$ are also referred to as
\emph{fibre maps} (defined on the \emph{fibres} $\{\theta\}\times [0,1]$).

An \emph{invariant graph} of $\eqref{skew_product}$ is a measurable function
$\varphi: \T^D\rightarrow [0,1]$ which satisfies
\begin{equation*}
    F_{\theta}(\varphi(\theta)) = \varphi(\theta + \rho) \qquad \textrm{for all } \theta\in\mathbb{T}^{D}.
\end{equation*}
From an intuitive perspective, invariant graphs are to be seen as non-autonomous
fixed points of \eqref{skew_product}.\footnote{Observe that due to the
  minimality of $\rho$, \eqref{skew_product} does not allow for actual fixed
  points.}  This idea is the basis for a bifurcation theory of invariant graphs,
see \cite{nunez/obaya:2007,AnagnostopoulouJaeger2012SaddleNodes}.  Independently
of this analogy, invariant graphs of qpf monotone interval maps are key to
understanding the dynamics of \eqref{skew_product} due to their intimate
relationship with the invariant sets and ergodic measures.

Every invariant graph $\phi$ comes with an ergodic measure $\mu_\phi$ where
$\mu_\phi(A)=\textrm{Leb}_{\T^D}(\phi^{-1}A)$ for each measurable $A\ssq
\T^D\times [0,1]$ and likewise, to each ergodic measure $\mu$ of
\eqref{skew_product} there is an invariant graph $\phi$ with $\mu=\mu_\phi$
\cite{Furstenberg1961,Arnold1998}.  Moreover, given an \emph{invariant set}
$A\subseteq \mathbb{T}^{D}\times [0,1]$ (that is, $A$ is closed and $F(A)=A$),
let
\begin{align*}
    \varphi_A^{+}(\theta)=\sup\{x\in [0,1]\:(\theta,x)\in A\} \qquad \text{and} \qquad 
    \varphi_A^{-}(\theta)=\inf\{x\in [0,1]\:(\theta,x)\in A\}
\end{align*}
for each $\theta\in \T^D$.  Then $\varphi_A^+$ and $\phi_A^-$ define the
so-called \emph{upper} and \emph{lower boundary graphs} of $A$ which are
invariant and---due to the compactness of $A$---upper and lower semi-continuous,
respectively \cite{stark:2003}.  Of particular relevance for us will be the
upper boundary graph of the \emph{global attractor}
\[
 \bigcap_{n\in\N}F^n(\T^D\times [0,1]),
\]
which we simply denote by $\phi^+$.

The long-term behaviour of orbits near an invariant graph $\phi$ is largely
characterized by its \emph{Lyapunov exponent}
\begin{align*}
    \lambda(\varphi)=\int_{\mathbb{T}^{D}}\log F^{\prime}_{\theta}(\varphi(\theta))d\theta,
\end{align*}
provided this integral exists.

If $\lambda(\varphi)>0$, then $\phi$ is repelling; if $\lambda(\varphi)<0$, then
$\varphi$ is attracting and $\mu_\phi$ is a physical measure, see
\cite{jaeger:2003} for the details.  Here, by \emph{physical measure}, we refer
to an $F$-invariant ergodic measure $\mathbb P$ for which there is a positive
Lebesgue measure set $V\ssq \T^D\times[0,1]$ such that for every continuous
observable $f: \T^D\times [0,1] \to \R$ and all $(\theta,x)\in V$
\[
 1/n\cdot \sum_{\ell=0}^{n-1} f(F^\ell(\theta,x))=\int\! f \, d\mathbb P,
\]
see also \cite{Young2002}.  It is noteworthy that under suitable concavity
assumptions on $F_\theta$ (which are verified by \eqref{eq: F intro}),
\eqref{skew_product} has a unique physical measure given by the measure
$\mu_{\phi^+}$ on the upper boundary graph $\phi^+$ of the global attractor, see
\cite{keller:1996,AnagnostopoulouJaeger2012SaddleNodes}.

Observe that due to Birkhoff's Ergodic Theorem, $\lambda(\phi)$ equals the
Lyapunov exponent of the point $(\theta,\phi(\theta))$ for
$\textrm{Leb}_{\T^D}$-almost every $\theta$ (equivalently: for $\mu_\phi$-almost
every $(\theta,x)$) since
\begin{align*}
   \lambda(\theta,\varphi(\theta))=
   \lim_{n\rightarrow\infty}\frac{1}{n} \operatorname{log}(F^{n}_{\theta})'(\varphi(\theta))=\lim_{n\rightarrow\infty}\frac{1}{n}\sum_{\ell=0}^{n-1}\operatorname{log} F^{\prime}_{\rho^{\ell}(\theta)}(\varphi(\rho^{\ell}(\theta)))
\end{align*}
for $\textrm{Leb}_{\T^D}$-almost every $\theta\in\mathbb{T}^{D}$.  Note that on
the left-hand side of the above equation, we made use of the customary notation
\begin{align}\label{eq: dynamical composition}
F^n_\theta(x)=\pi_2\circ F^n(\theta,x)=F_{\theta+(n-1)\rho}\circ \ldots \circ F_{\theta+\rho}\circ F_\theta(x),
\end{align}
where $\pi_2\:\T^D\times [0,1]\to [0,1]$ denotes the projection to the second coordinate.

\begin{remark}
  \label{r.zero_line_exponent} Note that for the model (\ref{eq: F intro})
  and the {\em a priori} invariant graph $\psi=0$ given by zero line, we have
  that $F_\theta'(\psi(\theta))=F_\theta'(0)=\kappa\cdot g(\theta)$, so that
  \[
  \lambda(\psi)=\log\kappa 0 \int_{\T^D} \log g(\theta)\ d\theta = \log\kappa -
  \log\kappa_0
  \]
  in this case. Hence, the zero line is repelling for all $\kappa>\kappa_0$, and
  pointwise Lyapunov exponents on this line are positive almost surely (with
  respect to the Lebesgue measure on $\T^D\times\{0\}$).
\end{remark}

As we will discuss below, the unique physical measure of \eqref{eq: F intro} is
given by $\mu_{\phi^+}$, where $\phi^+$ is the upper boundary graph of the
global attractor.  Therefore, a big part of the proof of Theorem~\ref{thm: main}
boils down to analysing $\phi^+$ in considerable detail.  In that context, we
will utilize the obvious fact that $\phi^+$ is the pointwise limit of the
sequence of \emph{iterated upper boundary lines} $(\phi_n)_{n\in \N_{\geq 0}}$,
where
\begin{align}\label{eq: defn iterated boundary lines}
 \phi_n\:\T^D\to [0,1],\qquad \theta\mapsto F_{\theta-n\rho}^n(1).
\end{align}
Note that the graph of $\phi_n$ coincides with $F^n(\T^D\times \{1\})$ (recall
the notation from \eqref{eq: dynamical composition}).  It is further easy to see
(and important to note) that the monotonicity of the fibre maps $F_\theta$
implies $\phi_{n+1}\leq \phi_n$ for all $n\in\N$. 

For the convenience of the reader, we close this section with a brief
description of the invariant graphs of \eqref{eq: F intro}.  While this
description will help to develop an intuition for the dynamics of \eqref{eq: F
  intro} and, more broadly, for the results discussed in this note, it is
strictly speaking not a prerequisite for the discussion in Section~\ref{sec:
  pinched systems} and Section~\ref{sec: proof of main result}.  For simplicity,
we may assume that $D=1$ in the remainder of this section.

It is immediate that independently of the value of $\kappa$, one invariant graph
of \eqref{eq: F intro} is given by the $0$-line (which just happens to be the
lower boundary graph of the global attractor).  Let us denote this graph by
$\phi^-$.  By direct computation, one can obtain that $\lambda(\phi^-)=\log
\kappa-\log 2$.

Clearly, if $\phi^-$ equals the upper boundary graph $\phi^+$ of the global attractor, then $\phi^-$ is the only invariant graph of $F_\kappa$.
However, with help of the iterated upper boundary lines, one can show that
$\lambda(\phi^+)\leq 0$, see \cite{jaeger:2003}.
Accordingly, if $\kappa>2$, the $0$-line $\phi^-$ is $\textrm{Leb}_{\T^1}$-almost surely distinct from $\phi^+$.

In other words, $F_{\kappa}$ has at least two invariant graphs if $\kappa>2$.
Moreover, just as concavity of interval maps implies the existence of at most
two fixed points (one of which is attracting and one of which is repelling), one
can show that the concavity of the fibre maps of $F_\kappa$ implies that
$\phi^-$ and $\phi^+$ are the only invariant graphs (and further,
$\lambda(\phi^-)>0>\lambda(\phi^+)$), see
\cite{keller:1996,AnagnostopoulouJaeger2012SaddleNodes}.\footnote{Note that
  accordingly, the physical measure $\mathbb P$ in the introduction has to
  coincide with $\mu_{\phi^+}$.}

Now, since $F(0,x)=0$, we have that $\phi^+$ necessarily intersects the $0$-line
along the orbit of $(\rho,0)$ which is, by minimality of $\rho$, dense in
$\T^D\times\{0\}$.  Therefore, while $\phi^+$ is upper-semicontinuous (as the
upper boundary graph of the global attractor) it clearly is not continuous and
$\phi^+$ is referred to as a \emph{strange non-chaotic attractor}, see
Figure~\ref{fig: sna pinched} for a plot of $\phi^+$.
\begin{figure}[!htb]
\centering
   \includegraphics[scale=0.35]{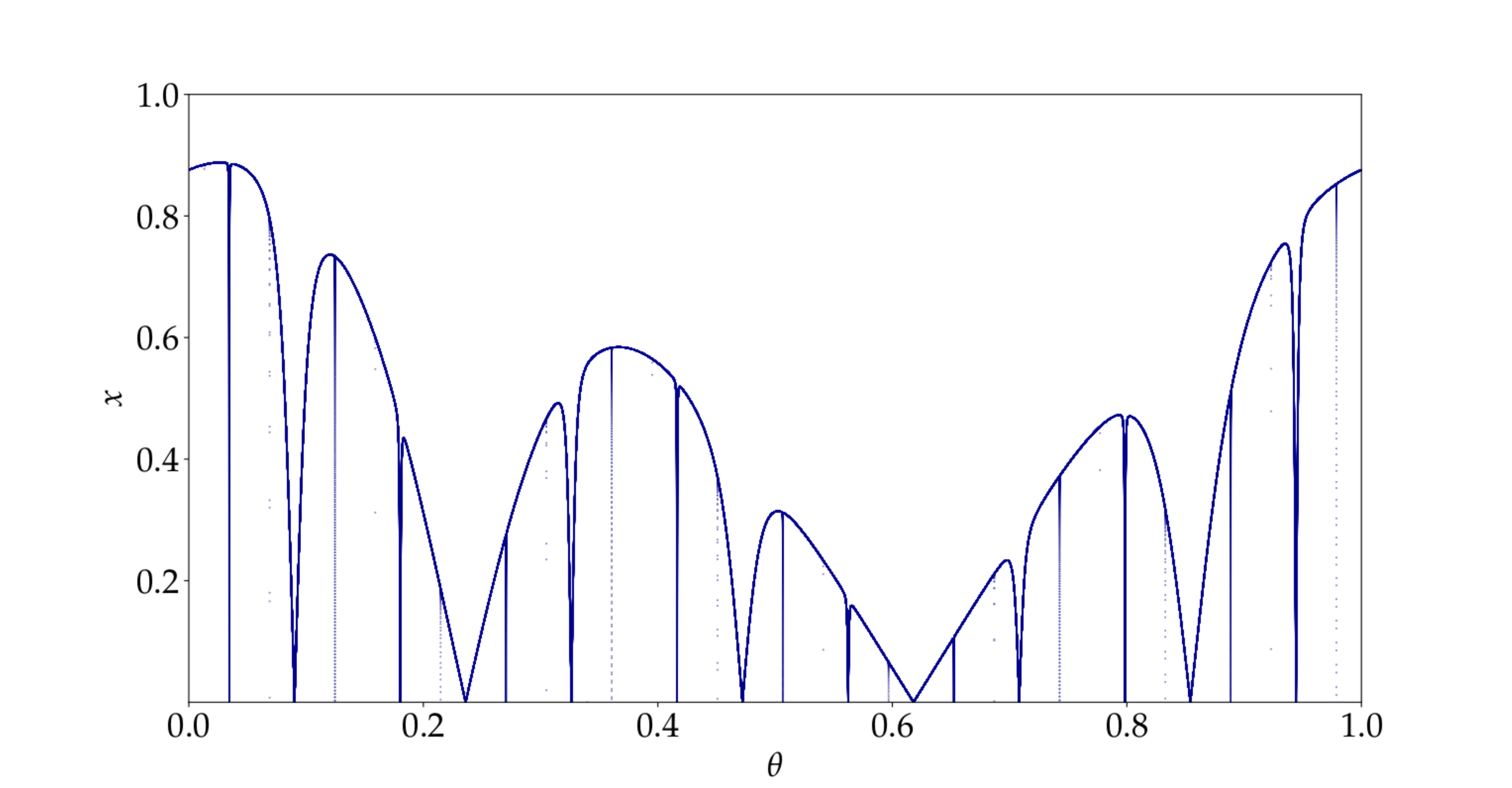} 
 \caption[Strange non-chaotic attractor in pinched systems]{The SNA $\phi^+$ of
   the parameter family $(\theta, x)\mapsto(\theta +\rho, \tanh(\kappa
   x)\cdot\sin(\pi\theta))$ with $\kappa=3$ and $\rho$ the golden mean.  The
   points in the above plot are exactly the initial conditions used to estimate
   $p_{\kappa,N}$ in Figure~\ref{fig: FTLE}.} \label{fig: sna pinched}
\end{figure}

\section{Pinched skew-product systems}\label{sec: pinched systems}
In this section, we specify the class of skew products within which we derive
asymptotic estimates on the probability of positive finite-time Lyapunov
exponents.  For later reference, we repeat some of the assumptions from the
previous section.  By $\mathcal{F}$, we refer to the class of quasiperiodically
forced monotone interval maps of the form
\begin{align*}
    F:\mathbb{T}^{D}\times [0,1]\rightarrow \mathbb{T}^{D}\times [0,1], \quad 
    (\theta, x) \mapsto (\rho(\theta),~ F_{\theta}(x)),
\end{align*}
which satisfy
\begin{description}
\item[{($\mathcal{F}_{1}$)}] the fibre maps $F_{\theta}$ are non-decreasing;
\item[{($\mathcal{F}_{2}$)}] the fibre maps $F_{\theta}$ are differentiable and
  $(\theta, x)\mapsto F_{\theta}^{\prime}(x)$ is continuous on
  $\mathbb{T}^{D}\times I$;
\item[{($\mathcal{F}_{3}$)}] $F$ is \emph{pinched}, that is, there is $\theta_*\in \T^D$ with $F_{\theta_*}(x)=0$ for all $x\in [0,1]$;
\item[{($\mathcal{F}_{4}$)}] $F_{\theta}(0)= 0$ for all $\theta \in \mathbb{T}^D$ (invariance of the 0-line).
\end{description}

Besides the qualitative assumptions $(\mc F_1)$--$(\mc F_4)$, we need a number of quantitative assumptions.
Let $F\in\mathcal{F}$ and assume that there exist parameters $\alpha>2,~\beta>0,~\gamma>0$ and $L_{0}\in(0, 1)$ such that for all $\theta\in\mathbb{T}^{D}$, the following holds.
\begin{align}
\label{f1}
\tag{F1}
\left|F_{\theta}(x)-F_{\theta}(y)\right|\leq \alpha\left|x-y\right| \qquad &\text{for all } x,y\in[0,1],\\
\label{f2}
\tag{F2}
\left|F_{\theta}(x)-F_{\theta}(y)\right|\leq \alpha^{-\gamma}\left|x-y\right| \qquad &\text{for all } x, y\in [L_{0}, 1],\\
\label{f3}
\tag{F3}
\left|F_{\theta}(x)-F_{\theta^{\prime}}(x)\right|\leq \beta\, d(\theta, \theta^{\prime}) \qquad &\text{for all } x\in [0, 1].
\end{align}
In particular, \eqref{f2} implies that the fibre maps $F_\theta$ are contracting in $[L_0,1]$. 

While the above assumptions determine $F$ in the fibres, we need some additional control over the forcing on $\T^D$.
To that end, we assume that the rotation vector $\rho\in\mathbb{T}^{D}$ is \emph{Diophantine}.
More specifically, setting $\tau_{n}=\rho^{n}(\theta_{*})=\theta_{*}+n\rho$ (the $n^{th}$-iterate of the \emph{pinched point} $\theta_{*}$), we assume that there are constants $c>0$ and $d>1$ such that
\begin{align}
\label{f4}
\tag{F4}
    d(\tau_{n},\theta_{*})\geq c\cdot n^{-d} \qquad \text{for all } n\in \mathbb{N}.
\end{align}

Finally, bringing the behaviour along the fibres and the dynamics on the base $\T^D$ together, we assume that there are constants $m\in\mathbb{N},~a>1$ and $0<b<1$ with 
\begin{align}
\label{f5}
\tag{F5}
    m>22\left(1+\frac{1}{\gamma}\right),\\
 \label{f6}
\tag{F6}
    a\geq (m+1)^{d},\\
\label{f7}
\tag{F7}
   b\leq c,\\
\label{f8}
\tag{F8}
     d(\tau_{n},\theta_{*})>b\qquad \text{for all } n\in\{1,\cdots,m-1\}
\end{align}
such that
 \begin{align}
\label{f9}
\tag{F9}
    F_{\theta}(x)\geq\min\left\{2L_{0},ax\right\}\cdot\min\left\{1, \frac{2}{b}d(\theta,\theta_{*})\right\} \qquad &\text{for all } (\theta, x)\in\mathbb{T}^{D}\times [0,1].
\end{align}
Our analysis of positive finite-time Lyapunov exponents will take place within the class
\begin{align*}
 \mc F^\ast=\{F\in \mc F\: F\text{ satisfies  \eqref{f1}--\eqref{f9}}\}.
\end{align*}
Instead of the abstract description of $\mc F^\ast$ given above, readers may simply
think of the system given in \eqref{eq: F intro} (for large $\kappa$) in all of the following.
This is justified by the next statement.
\begin{lem}[{see \cite[Lemma 4.2]{GroegerJaeger2013SNADimensions}}]\label{lem:1}
Consider $F_{\kappa}$ as in \eqref{eq: F intro} and let $\rho$ satisfy the
Diophantine condition \eqref{f4} for some $c>0$ and $d>1$.  There exists a
constant $\kappa_{0}=\kappa_{0}(c,d, D)$ such that for all
$\kappa\geq\kappa_{0}$, the map $F_{\kappa}$ satisfies \eqref{f1}--\eqref{f9}
(with appropriately chosen constants $\alpha,\gamma,\beta,L_0,m,a,b,c$).
\end{lem}

Note that as $[0,1]\ni x\mapsto F_\theta'(x)$ is continuous for each $\theta \in \T^D$ (due to $(\mathcal{F}_{2})$), the mean value theorem and \eqref{f9} imply
\begin{align}
\label{f10}
\tag{F10}
F_\theta'(0)\geq a\cdot \min\left\{1,\frac2b d(\theta,\theta_*)\right\} \qquad &\text{for all } \theta\in \T^D.
\end{align}
While \eqref{f10} yields the existence of positive finite-time Lyapunov exponents on the zero line (see Lemma~\ref{lem: positive lyapunov exponent on zero line} below), in order to ensure big enough lower bounds on the probability of positive finite-time Lyapunov exponents outside the zero line, we additionally assume that for all $\delta>0$ there is $x_\delta>0$ with
\begin{align}
\label{f11}
\tag{F11}
F_\theta'(x)\geq  (1-\delta)\cdot F_\theta'(0) \qquad &\text{for all } x\in[0,x_\delta] \text{ and all } \theta\in \T^D.
\end{align}
Clearly, this additional assumption is satisfied by \eqref{eq: F intro} (for all $\kappa,D$ and $\rho$).

\section{Rigorous bounds on the probability of positive finite-time Lyapunov exponents}\label{sec: proof of main result}
In this section, we show that within the class of pinched skew-products,
the $\mu_{\phi^+}$-measure of points $(\theta,x)$ with $\lambda_N(\theta,x)\geq 0$ decays exponentially as $N\to\infty$.\footnote{Recall that $\phi^+$ refers to the upper boundary graph 
of the global attractor, see Section~\ref{preliminaries}.}

We start by deriving a lower bound for this probability.  To that end, we first
need to study the occurrence of positive finite-time Lyapunov exponents on the
zero line.  Due to \eqref{f10}, this essentially amounts to analysing the
frequency of visits of points $\theta\in \T^D$ to the vicinity of $\theta_*$.

For $j\in \N$, set
\begin{align}\label{eq: defn R and r}
    r_{j}=\frac{b}{2}a^{-(j-1)}\qquad \text{and} \qquad
    R_{j}=\frac{b}{2}a^{\frac{-(j-1)}{m}}.
\end{align}

\begin{prop}
Suppose \eqref{f4}--\eqref{f8} are satisfied.
 Then, for $n\in\{1,\ldots,m-1\}$, we have
\[
 B_{r_1}(\theta_*)\cap (B_{r_1}(\theta_*)+n\rho)= \emptyset.
\]
Similarly, for $j\geq 2$ and $n\in\{1,\ldots, (m+1)^{(j-1)}\}$, we have
\[
 B_{r_j}(\theta_*)\cap (B_{r_j}(\theta_*)+n\rho)= \emptyset.
\]
\end{prop}
\begin{proof}
We only discuss $j\geq 2$.
With \eqref{f8}, the other case is obvious.

Suppose $B_{r_j}(\theta_*)\cap (B_{r_j}(\theta_*)+n\rho)\neq \emptyset$ for some
$n$, that is, $d(\theta_*,\tau_n)<2r_j=ba^{-(j-1)}$.  Note that \eqref{f4} gives
$d(\theta_*,\tau_n)\geq c\cdot n^{-d}$.  Therefore, $ba^{-(j-1)}> c\cdot n^{-d}$
and thus, $n> (c/b)^{1/d}\cdot a^{(j-1)/d}\geq (m+1)^{j-1}$, where we used
\eqref{f6} and \eqref{f7} in the last step.
\end{proof}
This immediately gives
\begin{cor}
 Assume \eqref{f4}--\eqref{f8} and let $\theta\in \T^D$.  Suppose $n_1<n_2 \in
 \N$ are such that $\theta+n_1\rho\in B_{r_j}(\theta_*)$ and $\theta+n_2\rho\in
 B_{r_j}(\theta_*)$.  If $j=1$, then $n_2-n_1\geq m$ and if $j\geq 2$, then
 $n_2-n_1\geq (m+1)^{j-1}$.
\end{cor}
\begin{lem}\label{lem: positive lyapunov exponent on zero line}
 Suppose $F\in \mc F^*$ and $N\in \N$.  For each $\theta\in B_{r_N}(\theta_*)$,
 we have
 \[\lambda_N(\theta+\rho,0)\geq 1/2 \cdot\log a.\]
\end{lem}
\begin{proof}
 Set $\Delta_j=(m+1)^{j-1}$.  By the previous corollary, given $\theta$ as in
 the assumptions and $2\leq j\leq N$, we have
 \begin{align*}
 \#\{\ell \in \{1,\ldots,N\}\:\theta+\ell \rho\in B_{r_j}(\theta_*)\}\leq \lfloor N/\Delta_j\rfloor
 \end{align*}
 and
 \begin{align*}
 \#\{\ell \in \{1,\ldots,N\}\:\theta+\ell \rho\in B_{r_1}(\theta_*)\}\leq \lfloor N/m\rfloor.
 \end{align*}
  Note further that for $j\geq 0$, \eqref{f10} gives
\begin{align*}
 F_{\theta+\ell\rho}'(0)\geq   a\cdot \frac2b r_{j+1}=a^{-(j-1)} \qquad \text{whenever } \theta+\ell\rho\in B_{r_j}(\theta_*)\setminus B_{r_{j+1}}(\theta_*),
\end{align*}
where---for notational convenience---$r_0=\sqrt{D}$ and hence
$B_{r_0}(\theta_*)=\T^D$.  We therefore have
\begin{align*}
 &\lambda_N(\theta+\rho,0) =1/N\cdot \sum_{\ell=1}^N\log
  F'_{\theta+\ell\rho}(0)\geq 1/N\cdot \sum_{\ell=1}^N\sum_{j\geq 0} \log
  a^{-(j-1)} \cdot \mathbf{1}_{B_{r_j}(\theta_*)\setminus
    B_{r_{j+1}}(\theta_*)}(\theta+\ell \rho)\\ &= 1/N\cdot
  \sum_{\ell=1}^N\sum_{j\geq 0} (1-j)\log a \cdot
  \mathbf{1}_{B_{r_j}(\theta_*)\setminus B_{r_{j+1}}(\theta_*)}(\theta+\ell
  \rho) \\ & =\log a -\log a\cdot 1/N\cdot \sum_{j\geq 1}j\cdot \sum_{\ell=1}^N
  \mathbf{1}_{B_{r_j}(\theta_*)\setminus B_{r_{j+1}}(\theta_*)}(\theta+\ell
  \rho) \\ & \geq \log a -\log a\cdot 1/N\cdot (\lfloor N/m\rfloor+\sum_{j\geq
    2} j\cdot \lfloor N/\Delta_j\rfloor) \geq \log a -\log a \cdot
  (1/m+\sum_{j\geq 2} j/\Delta_j)\\ &= \log a -\log a \cdot (1/m+\sum_{j\geq 2}
  j/(m+1)^{j-1})\geq 1/2\cdot \log a,
\end{align*}
where we used \eqref{f5} in the last step.
\end{proof}
In order to prove the lower bound in Theorem~\ref{thm: main}, it remains to show
that the positive finite-time Lyapunov exponents on the zero line are observable
not only \emph{on} but already \emph{close to} the zero line.  This is what the
proof of the next statement is about.
\begin{theorem}\label{thm: lower bound on probabilities}
  Suppose $F\in \mc F^*$ satisfies \eqref{f11}.  Then there is $\gamma_+>0$ such
  that for all $N\in \N$
 \[
 \mu_{\phi^+}\{(\theta,x)\in\mathbb{T}^D\times [0,1]\:\lambda_{N}(\theta,x)\geq
 0\}\geq e^{-\gamma_+ N}.
 \]
\end{theorem}
\begin{proof}
Choose $\delta>0$ small enough such that
\begin{align}\label{eq: derivative in strip around zero line}
 \log (1-\delta)>-(\log a)/4
\end{align}
and let $x_\delta$ be such that \eqref{f11} holds true.  Without loss of
generality, we may assume that $x_\delta\leq \beta b/2$ (with $\beta$ from
\eqref{f3}).  For $N\in \N$, set $\tilde r_N=x_\delta/\beta \cdot
\alpha^{-(N-1)}$.  Observe that $\alpha\geq a$ (because of \eqref{f1} and
\eqref{f9}) so that $\tilde r_N\leq r_N$ for all $N$.  We first show that for
$\theta \in B_{\tilde r_N}(\theta_*)$ and $j=1,\ldots,N$, we have
$\phi^+(\theta+j\rho)\leq x_\delta$.

To that end, observe that the monotonicity of the sequence of the iterated upper
boundary lines $(\phi_n)_{n\in \N}$ (recall \eqref{eq: defn iterated boundary
  lines}) and \eqref{f3} yield
\[
 \phi^+(\theta+\rho)=\lim_{n\to\infty} \phi_n(\theta+\rho)\leq F_\theta(1)\leq \beta \cdot d(\theta,\theta_*).
\]
Therefore, given $\theta \in B_{\tilde r_N}(\theta_*)$ and $j=1,\ldots,N$, we
have---due to $\mathcal{F}_{1}$ and \eqref{f1}---that
\[
\phi^+(\theta+j\rho)= F^{j-1}_{\theta+\rho}\left(\phi^+(\theta+\rho)\right)\leq
F^{j-1}_{\theta+\rho}(\beta\cdot d(\theta,\theta_*))\leq\alpha^{j-1}\beta \cdot
d(\theta,\theta_*)\leq x_\delta.
\]
As a consequence, Lemma~\ref{lem: positive lyapunov exponent on zero line} and
\eqref{eq: derivative in strip around zero line} in conjunction with \eqref{f11}
give that $\lambda_N(\theta+\rho,x)\geq (\log a)/4$ for all
$(\theta,\phi^{+}(\theta))$ with $\theta\in B_{\tilde r_N}(\theta_*)$.  The
statement follows.
\end{proof}

Having thus seen how within $\mc F^*$ (under the additional assumption of
$\eqref{f11}$) the probability of positive finite-time Lyapunov exponents decays
at most exponentially, we next come to show that this decay is, in fact, not
slower than exponential.

Before we turn to the rigorous analysis, we briefly explain its idea on an
intuitive level.  First, note that \eqref{f2} implies that above $L_0$, fibres
are contracted---we emphasize this fact by calling $\T^D\times [L_0,1]$ the
\emph{contracting region}.  In other words, visits to $\T^D\times [L_0,1]$
contribute negatively to the (finite-time) Lyapunov exponent of an orbit.
Second, \eqref{f9} enables us to control the number of times an orbit spends
outside of the contracting region.  Finally, since \eqref{f1} gives an upper
bound for the possible fibre-wise expansion, the control obtained through
\eqref{f9} enables us to ensure an overall contraction, that is, a negative
(finite-time) Lyapunov exponent, along \emph{most} finite orbits.

Let us specify this control in quantitative terms by collecting two auxiliary
statements from \cite{GroegerJaeger2013SNADimensions}.  Given
$\theta\in\mathbb{T}^{D}$ and $n\in\mathbb{N}$, let
$\theta_{k}:=\rho^{k-n}(\theta)$ and $x_{k}:=\varphi_k(\theta_{k})$ for $0\leq
k\leq n$.  Note that $\varphi_{k}(\theta_{k})=F^{k}_{\theta_{0}}(1)$ and
$\varphi_{n}(\theta)=F^{n-k}_{\theta_{k}}(x_{k})$.  Let \begin{align*}
  s^{n}_{k}:=\#\{k\leq j < n\:x_{j}<2L_{0} \}
\end{align*}
and set $s_{n}^{n}(\theta)=0$.  Recall the definition of $R_j$ in \eqref{eq:
  defn R and r}.
 \begin{lem}[{\cite[Lemma 4.6]{GroegerJaeger2013SNADimensions}}]\label{lem:2}
 \label{lemma2}
 Let $F\in\mathcal{F}^{*}$ and $q,~n\in\mathbb{N}$ with $n\geq mq+1$.  Suppose
 that $\theta\notin \bigcup\limits_{j=q}^{n}B_{R_{j}}(\tau_{j})$ and consider
 $t\geq mq$.  Then
 \begin{align*}
      s^{n}_{n-t}(\theta)\leq\frac{11t}{m}.
 \end{align*}
 \end{lem}
 
As discussed in Section~\ref{preliminaries}, the iterated upper boundary lines
$\phi_n$ approximate the graph $\phi^+$ whose measure $\mu_{\phi^+}$ we are
interested in.  The next statement effectively provides numerical bounds for
this approximation.
\begin{prop}[{\cite[Proposition 4.4]{GroegerJaeger2013SNADimensions}}]\label{prop1}
Let $F\in \mathcal{F}^{*}$, $q\in\mathbb{N}$ and
$\eta=\gamma-\frac{11}{m}(1+\gamma)>0$. Then, for $n\geq mq+1$ and $\theta\notin
\bigcup\limits_{j=q}^{n}B_{R_{j}}(\tau_{j})$, we have that
$\left|\varphi_{n}(\theta)-\varphi_{n-1}(\theta)\right|\leq
\alpha^{-\eta(n-1)}$.
\end{prop}
\begin{remark}
 Observe that $\eta>0$ due to \eqref{f5} and note that $\eta$ is independent of
 $q$.
\end{remark}

Clearly, Proposition~\ref{prop1} gives that if $k,n\in\mathbb{N}$ satisfy $mq\leq k<n$ and $\theta\notin \bigcup\limits_{j=q}^{n}B_{R_{j}}(\tau_{j})$, then 
\begin{align}\label{eq1}
\begin{split}
\left|\varphi_{n}(\theta)-\varphi_{k}(\theta)\right|
\leq
\sum_{i=k+1}^{n}\left|\varphi_{i}(\theta)-\varphi_{i-1}(\theta)\right|
\leq \sum_{i=k+1}^{n}\alpha^{-\eta(i-1)}
\leq\frac{\alpha^{-\eta k}}{1-\alpha^{-\eta}}.
\end{split}
\end{align}
With the above statements, we are now in a position to prove the upper bound in Theorem~\ref{thm: main}.
\begin{theorem}\label{thm: exponential decay wrt graph measure}
  Suppose $F\in \mc F^*$.
  Then there is $\gamma_->0$ such that for all $N\in \N$
 \[
 \mu_{\phi^+}\{(\theta,x)\in\mathbb{T}^D\times [0,1]\:\lambda_{N}(\theta,x)\geq 0\}\leq e^{-\gamma_- N}.
 \]
\end{theorem}
\begin{proof}
Note that it suffices to show the statement for sufficiently large $N$.

Let $N\in\mathbb{N}$ be given.  As $\phi^+$ is the pointwise limit of the
non-increasing sequence $\phi_n$ and due to the continuous differentiability of
the fibre maps (see $(\mathcal{F}_{2})$), we have that for each $\theta$
\begin{align*}
     \lambda_{N}(\theta_0,\varphi^{+}(\theta_0))&=\frac{1}{N}\sum_{k=0}^{N-1}\log \left|F_{\theta_{k}}^{\prime}(\varphi^{+}(\theta_{k}))\right|
      =  \lim_{n\rightarrow\infty}\frac{1}{N}\sum_{k=0}^{N-1}\log \left|F_{\theta_{k}}^{\prime}(\varphi_{n}(\theta_{k}))\right|,
\end{align*}
where---as above---$\theta_k=\rho^{k-N}(\theta)$.

In a first instance, our goal is to derive assumptions
on $\theta$ which ensure that the expression
$\frac{1}{N}\sum_{k=0}^{N-1}\log \left|F_{\theta_{k}}^{\prime}(\varphi_{n}(\theta_{k}))\right|$ is negative and bounded away from zero for $n\geq N$ and large enough $N$.
The statement then follows by showing that these assumptions are only violated in a set of exponentially small measure as $N\to \infty$.

We start by collecting a number of estimates which we will then combine to
obtain an upper bound for $\frac{1}{N}\sum_{k=0}^{N-1}\log
\left|F_{\theta_{k}}^{\prime}(\varphi_{n}(\theta_{k}))\right|$.  First, let
$\kappa\in\mathbb{N}$ be large enough such that $\kappa\geq m$ and $
\frac{\alpha^{-\eta \cdot \kappa}}{1-\alpha^{-\eta}}<L_{0}$.  Consider $k_0\in
\N$ with $k_0\geq \kappa q$ (for some $q\in \N$ which we may consider fixed for
now).  Then \eqref{eq1} gives that for every $n>k\geq k_0$ and
$\theta\notin\bigcup\limits_{j=q}^{n}B_{R_{j}}(\tau_{j})$, we obtain
$\left|\varphi_{n}(\theta)-\varphi_{k}(\theta)\right|<L_{0}$.  In particular, if
$\varphi_{k}(\theta)\geq2L_{0}$, then $\varphi_{n}(\theta)\geq L_{0}$.
Therefore, if $n\geq k$ and
$\theta_{k}\notin\bigcup\limits_{j=q}^{n}B_{R_{j}}(\tau_{j})$ for some $k\geq
k_0$ with $\phi_k(\theta_k)>2L_0$, \eqref{f2} yields
\begin{align}\label{eq: F prime in contracting region}
|F_{\theta_{k}}^{\prime}(\varphi_{n}(\theta_{k}))|\leq \alpha^{-\gamma}.
\end{align}

Second, observe that due to \eqref{f1}, we have
\begin{eqnarray}\label{birkhoff_sum}
 \sum_{k=0}^{N-1}\log
 \left|F_{\theta_{k}}^{\prime}(\varphi_{n}(\theta_{k}))\right|&\leq&\sum_{k=0}^{k_{0}+1}\log\alpha
 + \sum_{k=k_{0}}^{N-1}\log
 \left|F_{\theta_{k}}^{\prime}(\varphi_{n}(\theta_{k}))\right|.
 \end{eqnarray}

 Third, let $N_0=N_0(k_0)\in \N$ be the smallest integer such that
 $N_0-k_{0}\geq mk_0/\kappa \geq mq$.  Then, Lemma~\ref{lemma2} allows us to
 estimate the number of times for which $\varphi_{k}(\theta_{k})<2L_{0}$.
 Specifically, if $N\geq N_0$ and $k\in\{k_{0},\ldots,N-1\}$, we obtain for all
 $\theta\notin\bigcup\limits_{j=q}^{N}B_{R_{j}}(\tau_{j})$
 \begin{equation}\label{estimate}
    s_{k_{0}}^{N}(\theta) =s_{N-(N-k_{0})}^{N}(\theta)\leq \frac{11}{m}(N-k_{0}).
\end{equation} 

Observe that with \eqref{f1}, \eqref{estimate} and \eqref{eq: F prime in
  contracting region}, we get
\begin{align}
\begin{split}\label{right_sum}
&\sum_{k=k_{0}}^{N-1}\log\left|F_{\theta_{k}}^{\prime}(\varphi_{n}(\theta_{k}))\right|\leq
  s_{k_{0}}^{N}(\theta) \log \alpha -
  \left(N-k_0-s_{k_{0}}^{N}(\theta)\right)\gamma
  \log\alpha\\ &\leq\frac{11}{m}(N-k_{0}) \log\alpha
  -\gamma\cdot\left(N-k_{0}-\frac{11}{m}(N-k_{0})
  \right)\log\alpha\\ &=\left(\left(\gamma-\frac{11}{m}(1+\gamma)\right)k_{0}+\left(\frac{11}{m}(1+\gamma)-\gamma\right)N\right)\log\alpha
  = (\eta k_{0}-\eta N)\log\alpha
\end{split}
\end{align}
whenever $n\geq N\geq N_0$ and
$\theta_{k}\notin\bigcup\limits_{j=q}^{n}B_{R_{j}}(\tau_{j})$ for all
$k=k_0,\ldots,N$ and where $\eta$ is as in Proposition~\ref{prop1}.  Plugging
\eqref{right_sum} into \eqref{birkhoff_sum}, we obtain (under the same
assumptions as above)
\begin{align}\label{birkhoff2}
 \sum_{k=0}^{N-1}\log |F_{\theta_{k}}^{\prime}(\varphi_{n}(\theta_{k}))|&\leq
 (2+(\eta+1) k_{0}-\eta N)\log\alpha.
\end{align}
Now, as $\eta>0$, there is $\nu=\nu(\eta)>0$ such that for all $N\geq k_0/\nu$,
the right-hand side in \eqref{birkhoff2} is negative (so that for all $n\geq N$,
the left-hand side is negative and bounded away from $0$).  Note that we may
assume without loss of generality that $\nu$ is small enough to ensure $k_0/\nu
\geq N_0(k_0)$.

Set 
\[
B_{q,k_0,N}=\Big\{\theta\in
\mathbb{T}^{D}\:\theta_{k}\in\bigcup_{j=q}^{\infty}B_{R_{j}}(\tau_{j})
~\textrm{for some } k\in\{k_{0},\ldots, N\}\Big\}.
\]
Observe that \eqref{birkhoff2} holds whenever $\theta$ is in the complement of
the set $B_{q,k_0,N}$ (given $k_0\geq \kappa q$ and $n\geq N\geq N_0(k_0)$).
Note that
\begin{align*}
\begin{split}
&\operatorname{Leb}_{\mathbb{T}^{D}}(B_{q,k_0,N})\leq (N-k_0+1)\cdot
  \operatorname{Leb}_{\mathbb{T}^{D}}\left(\bigcup_{j=q}^{\infty}B_{R_{j}}(\tau_{j})\right)
  \\ & \leq
  (N-k_0+1)\cdot\zeta_{D}\cdot\left(\frac{b}{2}\right)^D\sum_{j=q}^{\infty}a^{\frac{-(j-1)D}{m}}
  \\ & =(N-k_0+1)\cdot\zeta_{D}\cdot\left(\frac{b}{2}\right)^D
  a^{-\frac{(q-1)D}{m}}\sum_{j=0}^{\infty}(a^{-\frac{D}{m}})^{j}
=(N-k_0+1)\cdot a^{-\frac{(q-1)D}{m}} c(D),
\end{split}
\end{align*}
where $\zeta_{D}$ denotes the $\operatorname{Leb}_{\mathbb{T}^{D}}$-measure of the $D$-dimensional unit ball and $c(D)$ simply collects all the terms in the above estimate which are independent of $q,\, k_0$ and $N$, that is, $c(D)=\zeta_{D}\cdot\left(\frac{b}{2}\right)^D \sum_{j=0}^{\infty}(a^{-\frac{D}{m}})^{j}$.

Now, set $k_0(N)=\lfloor\delta N\rfloor$ for some $\delta\in(0,\nu)$ and
$q(N)=\lfloor\eps N\rfloor$ for some $\eps\in(0,\delta/\kappa)$.  Note that for
large enough $N$, we have $k_0(N)\geq \kappa q(N)$ and $N\geq k_0(N)/\nu\geq
N_0[=N_0(k_0(N))]$.  Hence, for sufficiently large $N$, the above gives
\begin{align*}
\operatorname{Leb}_{\mathbb{T}^{D}}(\{\theta\in\mathbb{T}^D\:(\lambda_{N}(\theta,\varphi^{+}(\theta))\geq 0\})
&\leq 
\operatorname{Leb}_{\mathbb{T}^{D}}(B_{q(N),k_0(N),N})\\
&\leq
N\cdot a^{-\frac{(\lfloor \eps N\rfloor-1)D}{m}} c(D)
\leq a^{-\frac{\eps N D}{2m}}.
\end{align*}
The statement follows with $\gamma_-=(\log a)\cdot \eps D/2M$. 
\end{proof}

\end{document}